\documentclass[article]{siamart171218}

\usepackage{amsfonts}
\usepackage{graphicx}
\usepackage{epstopdf}
\usepackage{algorithmic}

\newsiamremark{remark}{Remark}
\newsiamremark{hypothesis}{Hypothesis}
\crefname{hypothesis}{Hypothesis}{Hypotheses}
\newsiamthm{claim}{Claim}

\headers{Optimal robustness of port-Hamiltonian systems}{V. Mehrmann and P. Van Dooren}

\title{Optimal robustness of port-Hamiltonian systems}

\author{Volker Mehrmann\footnotemark[2]
\and Paul Van Dooren\footnotemark[4]}
\usepackage{amsopn}

\DeclareMathOperator{\diag}{diag}

\newcommand{\eg}{e.\,g.,\ }
\newcommand{\ie}{i.\,e.,\ }
\newcommand{\M}{\mathcal{M}}
\newcommand{\T}{\mathcal{T}}
\renewcommand{\S}{\mathcal{S}}

\newcommand{\Hn}{ {\mathbb{H}_n} }   
\newcommand{\XWpd}{ {\mathbb{X}^{\raisebox{0.2em}{{\fontsize{3}{2}\selectfont $>$}}}} }  
\newcommand{\XWpdpd}{ {\mathbb{X}^{\raisebox{0.2em}{{\fontsize{3}{2}\selectfont $\gg$}}}} }  

\DeclareMathOperator{\rank}{rank}

\newcommand {\matr}      [1] {\left[\begin{array}{#1}}
\newcommand {\rix}          {\end{array}\right]}

\begin{document}

\maketitle

\begin{abstract}
	We construct optimally robust port-Hamiltonian realizations of a given rational transfer function that represents a passive system. We show that the realization with a maximal passivity radius is a normalized port-Hamiltonian one. Its computation is linked to a particular solution of a linear matrix inequality that defines passivity of the transfer function, and we provide an algorithm to
construct this optimal solution.
We also consider the problem of finding the nearest passive system to a given non-passive one and provide a simple but suboptimal solution. 
\end{abstract}

\footnotetext[2]{
	Institut f\"ur Mathematik MA 4-5, TU Berlin, Str.\ des 17.\ Juni 136,
	D-10623 Berlin, FRG.
	\url{mehrmann@math.tu-berlin.de}. Supported by {\it the German Federal Ministry of Education and Research BMBF within the project EiFer}
	and by {\it Deutsche Forschungsgemeinschaft},
through CRC TRR 154 'Mathematical Modelling, Simulation and Optimization using the Example of Gas Networks'.}
\footnotetext[4]{
	Department of Mathematical Engineering, Universit\'e catholique de Louvain, Louvain-La-Neuve, Belgium.
	\url{paul.vandooren@uclouvain.be}. Supported  by {\it Deutsche Forschungsgemeinschaft},
	through TRR 154 'Mathematical Modelling, Simulation and Optimization using the Example of Gas Networks'.
}

\section{Introduction}

We consider realizations of linear dynamical systems that are denoted as \emph{positive-real or passive} and their associated transfer functions. In particular, we study positive real transfer functions  which play a fundamental role in systems and control theory: they represent \eg spectral density functions of
stochastic processes and arise in spectral factorizations.
Positive-real transfer functions form a convex set, and this property has lead to the extensive use of convex optimization techniques in this area \cite{BoyEFB94}.
\emph{Passive} systems and their relationships with \emph{positive-real transfer functions} are well studied, starting with the works  \cite{Kal63,Pop73,Wil72a,Wil72b} and the topic has recently received considerable attention in the context of \emph{port-Hamiltonian (pH) system models}, \cite{Sch04,SchJ14}.

In this paper we show that in the set of continuous-time pH realizations of positive-real transfer functions, there is a subset that achieves optimal robustness, in the sense that their \emph{passivity radius} is maximal. Considering the \emph{Laplace transform} of the linear time-invariant system
\begin{equation} \label{statespace_c}
 \begin{array}{rcl} \dot x & = & Ax + B u,\ x(0)=0,\\
y&=& Cx+Du,
\end{array}
\end{equation}
denoted as $\M:=\{A,B,C,D\}$, the transfer function is given by
\begin{equation}\label{tf}
\mathcal T(s)= D+C(sI-A)^{-1}B.
\end{equation}
Here $u:\mathbb R\to\mathbb{C}^m$,   $x:\mathbb R\to \mathbb{C}^n$,  and  $y:\mathbb R\to\mathbb{C}^m$  are vector-valued functions denoting, respectively, the \emph{input}, \emph{state},
and \emph{output} of the system. Denoting real and complex $n$-vectors ($n\times m$ matrices) by $\mathbb R^n$, $\mathbb C^{n}$ ($\mathbb R^{n \times m}$, $\mathbb{C}^{n \times m}$), respectively, the coefficient matrices satisfy $A\in \mathbb{C}^{n \times n}$,   $B\in \mathbb{C}^{n \times m}$, $C\in \mathbb{C}^{m \times n}$, and  $D\in \mathbb{C}^{m \times m}$.

We restrict ourselves to systems which are \emph{minimal}, \ie the pair $(A,B)$ is \emph{controllable}  (for all $s\in \mathbb C$, $\rank \mbox{\small $[\,s I-A \; | \; B\,]$} =n$), and the pair $(A,C)$ is \emph{observable} (\ie $(A^\mathsf{H},C^\mathsf{H})$ is controllable). Here, the Hermitian transpose and the transpose of a vector or matrix $V$ is denoted by
$V^{\mathsf{H}}$ and $V^{\mathsf{T}}$, respectively, and the identity matrix is denoted by $I_n$ or $I$ if the dimension is clear.
Throughout this article we will use the following notation.
We denote the set of Hermitian matrices in $\mathbb{C}^{n \times n}$ by $\Hn$.
Positive definiteness (semidefiniteness) of  $A\in \Hn$ is denoted by $A>0$ ($A\geq 0$).
The real and imaginary parts of a complex matrix $Z$ are written as $\Re (Z)$ and $\Im (Z)$, respectively, and $\imath$ is the imaginary unit.
We consider functions over $\Hn$, which is a vector space if considered as a \emph{real} subspace of $\mathbb{R}^{n\times n}+\imath \mathbb{R}^{n\times n}$.

The paper is organized as follows. In Section~\ref{sec:PH} we characterize the classes of passive systems and  of port-Hamiltonian (pH) systems. We then show in Section~\ref{sec:passrad} the relevance of pH systems in estimating the passivity radius of passive systems and construct in Section~\ref{sec:maxpass} realizations with optimal robustness margin for passivity. In Section \ref{sec:computing} we describe an algorithm to compute this optimal robustness margin.
In Sections \ref{sec:distance} and \ref{sec:stability} we show how to use these ideas to estimate the distance to the set of passive systems and to the set of stable systems.

\section{Passive systems and port-Hamiltonian realizations} \label{sec:PH}
The concepts of \emph{positive-realness} and \emph{passivity} are well studied. We briefly recall some important properties following \cite{Wil71}, and refer to the literature for a more detailed survey.
Consider a continuous-time system  as in (\ref{statespace_c})  and its transfer function $\mathcal T(s)$ as in \eqref{tf}. The transfer function $\mathcal T(s)$ is called {\em positive-real} \/if the
matrix-valued rational function
\[
\Phi(s):= \mathcal T^{\mathsf{H}}(-s) + \mathcal T(s)
\]
is positive semidefinite for $s$ on the imaginary axis, \ie $\Phi(\imath\omega)\geq 0$
for all $\omega\in \mathbb{R}$ and it is called \emph{strictly positive-real} if $\Phi(\imath \omega)>0 $ for all $\ \omega\in \mathbb{R}$.

For  $X \in \Hn$ and a system $\M=\{A,B,C,D\}$, we consider the matrix function
\begin{equation} \label{prls}
W(X,\M) := \left[
\begin{array}{cc}
 - A^{\mathsf{H}}X-X\,A & C^{\mathsf{H}} - X\,B \\
C- B^{\mathsf{H}}X & D^{\mathsf{H}}+D
\end{array}
\right],
\end{equation}
which we also denote by $W(X)$, when it is clear from the context which model we refer to.
If $\mathcal T(s)$ is  positive-real, then  there exists $X\in \Hn$ such that the \emph{linear matrix inequality (LMI)}
\begin{equation} \label{KYP-LMI}
W(X) \geq 0
\end{equation}
holds. In this context we will make frequent use of the sets
\begin{subequations}\label{LMIsolnsets}
\begin{align}
&\XWpd :=\left\{ X\in \Hn \left|   W(X) \geq 0,\ X >0 \right.\right\},
\label{XpdsolnWpsd} \\[1mm]
&\XWpdpd :=\left\{ X\in \Hn \left|   W(X) > 0,\ X >0 \right.\right\}. \label{XpdsolWpd}
\end{align}
\end{subequations}
A system $\M :=\left\{A,B,C,D\right\}$ is called \emph{passive} if there exists a state-dependent
\emph{storage function}, $\mathcal H(x)\geq 0$, such that for any $t_1,t_0\in \mathbb R$ with $t_1>t_0$,
 the \emph{dissipation inequality}
\begin{equation} \label{supply} \mathcal H(x(t_1))-\mathcal H(x(t_0)) \le \int_{t_0}^{t_1} \Re (y(t)^{\mathsf{H}}u(t)) \, dt
\end{equation}
holds.
If for all $t_1>t_0$, the inequality in \eqref{supply}
is strict then the system is  called \emph{strictly passive}.

If $D^{\mathsf{H}}+D$ is invertible, then
the minimum rank solutions of \eqref{KYP-LMI} in $\XWpd$
are those for which $\rank W(X) = \rank (D^{\mathsf{H}}+D)  = m$, which in turn is the case
if and only if the Schur complement of $D^{\mathsf{H}}+D$ in $W(X)$ is zero.  This Schur
complement is associated with the continuous-time \emph{algebraic Riccati equation (ARE)}
\begin{equation}
\mathsf{Ricc}(X) := -XA-A^{\mathsf{H}}X  -(C^{\mathsf{H}}-XB)(D^{\mathsf{H}}+D)^{-1}(C-B^{\mathsf{H}}X)=0.\label{riccatic}
\end{equation}
Solutions $X$ to (\ref{riccatic}) yield a spectral factorization of $\Phi(s)$, and each solution corresponds to an invariant subspace spanned by the columns of $U:=\matr{cc} I_n & -X^{\mathsf{T}} \rix^{\mathsf{T}}$
that remains invariant under multiplication with the \emph{Hamiltonian matrix}
\begin{equation}\label{HamMatrix}
H:=\matr{cc} A-B (D^{\mathsf{H}}+D)^{-1} C & - B (D^{\mathsf{H}}+D)^{-1} B^{\mathsf{H}} \\
C^{\mathsf{H}} (D^{\mathsf{H}}+D)^{-1} C & -(A-B (D^{\mathsf{H}}+D)^{-1} C)^{\mathsf{H}} \rix,
\end{equation}
\ie $U$ satisfies $HU=U A_F$ for a \emph{closed loop matrix} $A_{F}=A-BF$ with $F := (D^{\mathsf{H}}+D)^{-1}(C-B^{\mathsf{H}}X)$, see e.g., \cite{FreMX02}.

We can also associate with $\Phi$ a system pencil
\begin{equation} \label{statespace}
S(s) :=
\left[ \begin{array}{cc|c} 0 & A-sI_n & B \\
	A^{\mathsf{H}}+sI_n & 0 & C^{\mathsf{H}} \\ \hline B^{\mathsf{H}} & C & D^{\mathsf{H}}+D  \end{array} \right].
\end{equation}
Then the Schur complement of $S(s)$ is the transfer function
$\Phi(s)$ and the generalized eigenvalues of $S(s)$ are the zeros of $\Phi(s)$. The properties of the system can, however,  be checked in a much more numerically robust way using the pencil $S(s)$ rather than the matrix $H$. This is, in particular, true if $D^{\mathsf{H}}+D$ is singular or ill-conditioned, see \cite{BenLMV15} for a detailed analysis and appropriate algorithms.

A special class of realizations of passive systems is that of \emph{port-Hamiltonian systems}.
\begin{definition}\label{def:ph}
A linear time-invariant \emph{port-Hamiltonian (pH) system} has the state-space form
\begin{equation} \label{pH}
 \begin{array}{rcl} \dot x  & = & (J-R)Q x + (G-K) u,\\
y&=& (G+K)^{\mathsf{H}}Q x+(S+N)u,
\end{array}
\end{equation}
and the system matrices satisfy the symmetry conditions
\[
\mathcal V:= \left[ \begin{array}{cccc} J & G \\ -G^{\mathsf{H}} & N \end{array} \right]=-\mathcal V^{\mathsf{H}},\
\mathcal W:= \left[ \begin{array}{cccc} R & K \\ K^{\mathsf{H}} & S \end{array} \right] =\mathcal W^{\mathsf{H}}\geq 0, \  Q=Q^{\mathsf{H}} >0.
\]
\end{definition}
Port-Hamiltonian systems were introduced from a different point of view in \cite{Sch04}, but they also have a storage function and satisfy a dissipation inequality, and hence they are passive. Thus,  there must  be a coordinate transformation between a passive system and a representation ({\ref{pH}) as a pH system. We briefly recall the construction of such a possible transformation.

Consider a minimal state-space model  $\M:=\{A,B,C,D\}$ of a passive linear time-invariant system and let  $X\in \XWpd$ be a solution of the LMI \eqref{KYP-LMI}.
We then use a symmetric factorization $X= T^{\mathsf{H}}T$, which implies the invertibility of $T$, and define a new realization
\[
\{A_T,B_T,C_T,D_T\} := \{TAT^{-1}, TB, CT^{-1}, D \}
\]
so that
\begin{eqnarray}  \nonumber
&&\left[ \begin{array}{cccc} T^{-\mathsf{H}} & 0\\ 0 & I_m
\end{array}
\right]
\left[ \begin{array}{cccc} -A^{\mathsf{H}}X-XA & C^{\mathsf{H}}-XB \\ C-B^{\mathsf{H}}X & D^{\mathsf{H}}+D
\end{array}
\right]
\left[ \begin{array}{cccc} T^{-1} & 0\\ 0 & I_m
\end{array}
\right] \\
\label{PH}
&& \qquad =
\left[ \begin{array}{cccc}-A_T & -B_T \\ C_T & D_T
\end{array}
\right]+
\left[ \begin{array}{cccc} -A_T^{\mathsf{H}} & C^{\mathsf{H}}_T \\ -B^{\mathsf{H}}_T & D^{\mathsf{H}}_T
\end{array}
\right]
\geq 0.
\end{eqnarray}
We can then use the Hermitian and skew-Hermitian part of the matrix
\[
 \S := \left[ \begin{array}{cccc}-A_T & -B_T \\ C_T & D_T
\end{array}
\right]
\]
to define the coefficients of a pH representation  via
\[
 \left[ \begin{array}{cccc} R & K \\ K^{\mathsf{H}} & S
\end{array}
\right] := \frac{{\S} +{\S}^{\mathsf{H}}}{2} \geq 0, \quad \left[ \begin{array}{cccc} J &  G \\ -G^{\mathsf{H}} & N
\end{array}
\right] :=  \frac{{\S} - {\S}^{\mathsf{H}}}{2}.
\]
This construction yields $Q=I_n$ because of the chosen factorization $X=T^{\mathsf{H}}T$. Note that the factor $T$ is unique up to a unitary factor $U$, since $T^{\mathsf{H}}U^{\mathsf{H}}UT=T^{\mathsf{H}}T$, but this factor $U$ will not affect the results described in this paper.

There is  a lot of freedom in the representation of the system, since we could have used any matrix $X$ from the set $\XWpd$, or we could have chosen a representation where $Q$ was not the identity matrix. In the remainder of this paper, we will restrict ourselves to pH models where $Q=I_n$. The freedom remaining is thus the choice of the matrix $X$ from $\XWpd$, which we will use to make the representation 'maximally'  robust or well-conditioned to perturbations.

\begin{remark}\label{rem2}
We stress  that when the model $\M$ is real, then all the definitions and properties discussed above still hold. Moreover, the sets $\XWpd$ and $\XWpdpd$ can be constrained to be
real without altering any of the results, since the real part $X_\Re$ of
a Hermitian matrix $X$ is symmetric and positive (semi-)definite whenever $X$ is Hermitian and (semi-)definite. When the model $\M$ is real it therefore follows that whenever $W(X)\ge 0$ (or $W(X) > 0$) then we also have that $W(X_\Re)\ge 0$ (or $W(X_\Re)>0$), and it then suffices to verify these conditions over the real symmetric matrices only. When doing that, the corresponding pH realizations will also be real. Finally, we point out that the extremal solutions $X_-$ and $X_+$ of the Riccati equations are also real when the model $\M$ is real.
\end{remark}

\section{The passivity radius}\label{sec:passrad}
Our goal to achieve robust pH representations of a passive system can be realized in different ways. A natural measure for this optimality is a large \emph{passivity radius}
$\rho_{\M}$, which is the smallest perturbation (in an appropriate norm) to the coefficients of a model $\M $ that makes the system non-passive. In this section we recall and extend a few results on passivity radii from \cite{BeaMV19} that will be employed in the next section.

Once we have determined a solution $X\in \XWpd$ to the LMI~(\ref{KYP-LMI}), we can determine the representations \eqref{pH} as discussed in Section~\ref{sec:PH} and the system is automatically passive (but not necessarily strictly passive). For each such representation we can determine the passivity radius and then choose the solution $X\in \XWpd$ which is most robust under perturbations ${\Delta_\M}$ of the model parameters $\M:=\{A,B,C,D\}$. This is a very suitable approach for perturbation analysis, since  as soon as we fix $X$, the matrix
\begin{equation} \label{hx}
W(X,\M) = \left[
\begin{array}{cc}
0 & C^{\mathsf{H}} \\
C & D^{\mathsf{H}}+D
\end{array}
\right]
- \left[\begin{array}{cc}
A^{\mathsf{H}}X + X\,A & X\,B \\
B^{\mathsf{H}}X & 0
\end{array}
\right]
\end{equation}
is linear in the unknowns $A,B,C,D$ and when we perturb the coefficients, then we preserve strict passivity as long as
\begin{eqnarray*}
 W(X,\M+\Delta_\M) &:=& \left[
\begin{array}{cc}
 0 & (C+\Delta_C)^{\mathsf{H}} \\
(C+\Delta_C) & (D+\Delta_D)^{\mathsf{H}}+(D+\Delta_D)
\end{array}
\right]\\
&& \qquad- \left[\begin{array}{cc}
(A+\Delta_A)^{\mathsf{H}}X + X\,(A+\Delta_A) & X\,(B+\Delta_B) \\
(B+\Delta_B)^{\mathsf{H}}X & 0
\end{array}
\right]>0.
\end{eqnarray*}

Hence, given $X\in \XWpdpd$, we can look for the smallest perturbation $\Delta_\M$ to our model $\M $ that makes $\det W(X,\M+\Delta_\M)=0$. To measure the size of the perturbation $\Delta_\M$ of a state space model $\M $ we will use the Frobenius norm or the spectral norm
\[
 \|\Delta_\M \|_F := \left \|\left[\begin{array}{ccc}
\Delta_A & \Delta_B \\
\Delta_C & \Delta_D
\end{array}\right] \right \|_F , \quad
 \|\Delta_\M \|_2 := \left \|\left[\begin{array}{ccc}
\Delta_A & \Delta_B \\
\Delta_C & \Delta_D
\end{array}\right] \right \|_2,
\]
and we use also the following \emph{$X$-passivity radius}, which was introduced in \cite{BeaMV19} and gives a bound for the usual passivity radius.
\begin{definition}
For $X\in \XWpdpd$ the  \emph{$X$-passivity radius} is defined as
\[
	\rho_\M(X):= \inf_{\Delta_\M\in \mathbb C^{n+m,n+m}}\left\{ \| \Delta_\M \| \; | \; \det W(X,\M+\Delta_\M) = 0\right\}.
\]
\end{definition}
Note that in order to compute $\rho_\M(X)$ for the model $\M $, we must have a matrix $X\in \XWpdpd$, since $W(X,\M)$  must be positive definite to start with and also $X$ should be positive definite to obtain a state-space transformation from it. The following relation between the $X$-passivity radius and the usual passivity radius was also given in \cite{BeaMV19}.
\begin{lemma}
	The passivity radius for a given model $\M$ satisfies
	\begin{equation}
\nonumber
	\rho_{\M}:= \sup_{X\in \XWpdpd}\inf_{\Delta_\M\in \mathbb C^{n+m,n+m}}\{\| \Delta_\M \| | \det W(X,\M+\Delta_\M)=0\}= \sup_{X\in \XWpdpd} \rho_{\M}(X).\label{passive}
	\end{equation}
\end{lemma}
We now provide an exact formula for the $X$-passivity radius based on a one-parameter optimization problem. For this, we rewrite the condition $W(X,\M+\Delta_\M)>0$ as
\begin{eqnarray}
&& \left[\begin{array}{cc} -X & 0 \\ 0 & I_m \end{array}\right]
\left[\begin{array}{cc} A + \Delta_A & B+\Delta_B \\ C+\Delta_C & D+\Delta_D \end{array}\right]\nonumber \\
&& \qquad +
\left[\begin{array}{cc} A^{\mathsf{H}}+\Delta_A^{\mathsf{H}} & C^{\mathsf{H}}+\Delta_C^{\mathsf{H}} \\ B^{\mathsf{H}}+\Delta_B^{\mathsf{H}} & D^{\mathsf{H}}+\Delta_D^{\mathsf{H}} \end{array}\right]
\left[\begin{array}{cc} -X & 0 \\ 0 & I_m \end{array}\right] >0. \label{hatXW}
\end{eqnarray}
Setting
\begin{equation} \label{defWhatX}
\hat W:=W(X), \; \hat X := \left[\begin{array}{cc} X & 0 \\ 0 & I_m \end{array}\right], \; \Delta_\S := \left[\begin{array}{cc} -\Delta_A & -\Delta_B \\ \Delta_C & \Delta_D \end{array}\right],
\end{equation}
inequality \eqref{hatXW} can be written as the LMI
\begin{equation} \label{WDelta}
W(X,\M+\Delta_\M)= \hat W+\hat X \Delta_\S + \Delta_\S^{\mathsf{H}} \hat X >0
\end{equation}
as long as the system is still passive. In order to violate this condition, we need to find the smallest $\Delta_\S$ such that $\det W(X,\M+\Delta_\M) =0$. The following theorem, based on results of \cite{OveV05} and \cite{BeaMV19}, gives the minimal perturbation $\Delta_\S$ in both the Frobenius norm and the spectral norm.  We point out that the definition of $\Delta_\S$ yields that $\|\Delta_\S\|_{2}=\|\Delta_\M\|_{2}$ and $\|\Delta_\S\|_{F}=\|\Delta_\M\|_{F}$.
\begin{theorem} \label{thm:mingamma}
	Consider the matrices $\hat X, \hat W$ in (\ref{defWhatX}) and the pointwise positive semidefinite matrix function
\begin{equation}\label{defmgamma}
M(\gamma):= \left[ \begin{array}{cc} \gamma \hat X \hat W^{-\frac{1}{2}} \\ \hat W^{-\frac{1}{2}} / \gamma\end{array}\right]
\left[ \begin{array}{cc} \gamma \hat W^{-\frac{1}{2}} \hat X & \hat W^{-\frac{1}{2}} / \gamma \end{array}\right]
\end{equation}
in the real parameter $\gamma$. Then the largest eigenvalue $\lambda_{\max}(M(\gamma))$ is a \emph{unimodal function} of $\gamma$, ({i.e.} it is first monotonically decreasing and then monotonically increasing with growing $\gamma$). At the minimizing value $\underline \gamma$,  $M(\underline{\gamma})$ has an eigenvector $z$, {i.e.}
\[
 M(\underline{\gamma}) z = \underline\lambda_{\max} z, \quad z:=\left[ \begin{array}{cc} u \\ v \end{array}\right],
 \]
where
$  \|u\|_2^2=\|v\|_2^2=1$.
The minimum norm perturbation $\Delta_\S$ is of rank $1$ and is given by $\Delta_\S=uv^{\mathsf{H}}/\underline{\lambda}_{\max}$. It has norm $1/\underline{\lambda}_{\max}$
both in the spectral norm and in the Frobenius norm.
\end{theorem}
In \cite{BeaMV19} also the following  simple bound for $\underline{\lambda}_{\max}$ was derived.
\begin{corollary}\label{cor:lev} Consider the matrices $\hat X, \hat W$ in (\ref{defWhatX}) and the pointwise positive semidefinite matrix function $M(\gamma)$ as in (\ref{defmgamma}). The largest eigenvalue of $M(\gamma)$ is also the largest eigenvalue of
\[
\gamma^2 \hat W^{-\frac{1}{2}} \hat X^2 \hat W^{-\frac{1}{2}} + \hat W^{-1}/\gamma^2.
\]
An upper bound for $\underline{\lambda}_{\max}$ is given by $\underline{\lambda}_{\max}\le \frac{2}{\alpha\beta}$ where $\alpha^2:=\lambda_{\min}(\hat W)$ and $\beta^2=\lambda_{\min}(\hat X^{-1}\hat W\hat X^{-1})$. The corresponding lower bound for $\| \Delta_\S \|_2$ and $\| \Delta_\S \|_F$ is given by
\[
  \rho_\M(X) = \min_{\gamma} \| \Delta_\S \|_2 =\min_{\gamma} \| \Delta_\S \|_F \ge \alpha\beta/2.
\]
\end{corollary}
The following theorem, also proven in \cite{BeaMV19}, constructs a rank one perturbation which makes the matrix  $W_{\Delta_\M}$ singular and therefore gives an upper bound for $\rho_M(X)$.
\begin{theorem}\label{thm:Xpassivity}
Let $\M=\{A,B,C,D\}$ be a given model and assume that we are given a matrix $X\in \XWpdpd$, then the $X$-passivity radius $\rho_\M(X)$ is bounded by
\[
 \alpha\beta/2 \le \rho_\M(X) \le  \alpha\beta/(1+|v^{\mathsf{H}}w|),
 \]
where $u$, $v$ and $w$ are vectors of norm $1$, satisfying
\[
\alpha^2:=\lambda_{\min}(\hat W),\beta^2=\lambda_{\min}(\hat X^{-1}\hat W\hat X^{-1}), \hat W^{-\frac{1}{2}}v=v/\alpha, \hat W^{-\frac{1}{2}} \hat X u=w/\beta.
\]
Moreover, if $v$ and $w$ are linearly dependent, then $\rho_\M(X)=\alpha\beta/2$.
\end{theorem}
The following corollary shows how  these results can be applied to pH systems.
\begin{corollary}\label{cor:xeqI} If for a given system $\mathcal M$ we have that $X=I_n$   then  the corresponding representation of the system is port-Hamiltonian, {i.e.}, it has the representation $\M:= \{J-R,G-K,G^{\mathsf{H}}+K^{\mathsf{H}},S+N\}$  and the X-passivity radius of this model is given by
 \[
 \rho_\M(I)=\frac12 \lambda_{\min}W(I,\M)= \lambda_{\min}\left[\begin{array}{cc} R & K \\ K^{\mathsf{H}} & S	 \end{array}\right].
 \]
\end{corollary}
\begin{proof}
    The proof follows directly from Theorem~\ref{thm:Xpassivity}, since under the given assumption we have  $\alpha=\beta$ and we can choose $u=w$.
\end{proof}
Finally, we also show that pH realizations always have a better $X$-passivity radius than models that are not in pH form.
\begin{theorem}\label{thm:pHoptimal}
Let $\M=\{A,B,C,D\}$ be a given model and let $X\in\XWpdpd$, then the port-Hamiltonian model
$\M_T=\{TAT^{-1},TB,CT^{-1},D\}$ constructed from any matrix $T$ such that $X=T^\mathsf{H}T$, has an $I$-passivity radius
$\rho_{\M_T}(I)$ which is at least as large as $\rho_\M(X)$.
\end{theorem}
\begin{proof}
	It follows from Corollary \ref{cor:lev} that $\rho_\M(X)$ satisfies
\[ \rho_\M^{-1}(X)=\min_{\gamma}\|\left[ \begin{array}{cc} \gamma \hat W^{-\frac{1}{2}}\hat X & \hat W^{-\frac{1}{2}} / \gamma\end{array}\right] \left[ \begin{array}{cc} \gamma \hat X \hat W^{-\frac{1}{2}} \\ \hat W^{-\frac{1}{2}} / \gamma\end{array}\right] \|_2,
\]
while $\rho_{\M_T}(I)$ satisfies
\begin{eqnarray*}
 \rho_{\M_T}^{-1}(I)&=&\min_{\gamma}\|\left[ \begin{array}{cc} \gamma \hat W^{-\frac{1}{2}}\hat X^{\frac{1}{2}} & \hat W^{-\frac{1}{2}}\hat X^{\frac{1}{2}} / \gamma\end{array}\right] \left[ \begin{array}{cc} \gamma \hat X^{\frac{1}{2}} \hat W^{-\frac{1}{2}} \\ \hat X^{\frac{1}{2}}\hat W^{-\frac{1}{2}} / \gamma\end{array}\right] \|_2\\
&=& \|\left[ \begin{array}{cc} \hat W^{-\frac{1}{2}}\hat X^{\frac{1}{2}} & \hat W^{-\frac{1}{2}}\hat X^{\frac{1}{2}} \end{array}\right] \left[ \begin{array}{cc}  \hat X^{\frac{1}{2}} \hat W^{-\frac{1}{2}} \\ \hat X^{\frac{1}{2}}\hat W^{-\frac{1}{2}} \end{array}\right] \|_2 = \| 2 \hat W^{-\frac{1}{2}} \hat X \hat W^{-\frac{1}{2}}\|_2.
\end{eqnarray*}
But the matrix inequality
\[ \left[ \begin{array}{cc} \gamma \hat W^{-\frac{1}{2}} & \hat W^{-\frac{1}{2}} / \gamma\end{array}\right] \left[ \begin{array}{cc} \hat X^2 & -\hat X \\ -\hat X & I_n \end{array}\right] \left[ \begin{array}{cc} \gamma \hat W^{-\frac{1}{2}} \\ \hat W^{-\frac{1}{2}} / \gamma\end{array}\right] \ge 0
\]
	implies that
\[
	\left[ \begin{array}{cc} \gamma \hat W^{-\frac{1}{2}} & \hat W^{-\frac{1}{2}} / \gamma\end{array}\right] \left[ \begin{array}{cc} \hat X^2 & 0 \\ 0 & I_n \end{array}\right] \left[ \begin{array}{cc} \gamma \hat W^{-\frac{1}{2}} \\ \hat W^{-\frac{1}{2}} / \gamma\end{array}\right] \ge
	2 \hat W^{-\frac{1}{2}} \hat X \hat W^{-\frac{1}{2}}
\]
for all values of $\gamma$, and therefore $\rho_{\M}^{-1}(X) \ge \rho_{\M_T}^{-1}(I)$ or $\rho_{\M}(X) \le \rho_{\M_T}(I)$.
Note also that any other factorization $X=(UT)^\mathsf{H}(UT)$ yields the same result, since $\rho_{\M_{(UT)}}(I)=\rho_{\M_T}(I)$.
\end{proof}

\section{Maximizing the passivity radius} \label{sec:maxpass}
The main goal of our paper is the maximization of the passivity radius over all pH representations of a passive system. For this we now have a closer look at the constrained LMI
\begin{equation} \label{xi}
W(X,\M) \ge \xi \diag(X,I_m)
\end{equation}
and obtain the following theorem.
\begin{theorem}  \label{thm:maxoverX}
Let $\M:=\{A,B,C,D\}$ be a minimal realization of a passive system, and let $X$ be any matrix in $\XWpd$. Then there is a unique $\xi^*(X)$ which is maximal for
the matrix inequality \eqref{xi}. Moreover, this value of $\xi^*(X)$ is also the passivity radius of the pH system $\M_T=\{TAT^{-1},TB,CT^{-1},D\}$, where $X=T^\mathsf{H}T$, in both the  spectral and Frobenius norm.
\end{theorem}
\begin{proof}
Every $X\in\XWpd$ is strictly positive definite, and can thus be factorized as $X=T^{\mathsf{H}}T$ with $\det T\neq 0$. Thus, we can define the transformed system $\M_T=\{TAT^{-1},TB,CT^{-1},D\}$. It is obvious that the passivity LMI $W(X,\M_T)\ge 0$ of the transformed system $\M_T$ is satisfied with $X=I_n$ and that it is related to the passivity LMI $W(X,\M)\ge 0$ of $\M$ via
\[
W(I,\M_T) :=  \left[\begin{array}{cc} T^{-{\mathsf{H}}} & 0 \\ 0 & I_m \end{array}\right] W(X,\M)  \left[\begin{array}{cc} T^{-1} & 0 \\ 0 & I_m \end{array}\right]\ge 0.
\]
It also follows that \eqref{xi} is satisfied if and only if $W(I,\M_T) \ge \xi I_{n+m}$ is satisfied. But the largest value $\xi^*(X)$ of
$\xi$ for which this holds is clearly equal to
\begin{eqnarray}  \nonumber \xi^*(X) & := & \max_\xi\left\{ \;\xi\; | \;  W(X,\M) \ge \xi  \diag(X,I_{m})\right\} \\  \nonumber
& = & \max_\xi\left\{ \;\xi\; | \;  W(I,\M_T) \ge \xi  I_{n+m}\right\} = \lambda_{\min} W(I,\M_T).
\end{eqnarray}
Since state-space transformations do not change the transfer function, it follows that $\M_T$ is a particular pH realization of the transfer function of $\M$
and that
\[
W(I,\M_T) = 2 \left[\begin{array}{cc} R & K^{\mathsf{H}} \\ K & S \end{array}\right]
\]
as in Definition~\ref{def:ph}.
Then, it follows from Corollary~\ref{cor:xeqI} that $\xi^*(X)$ is also equal to the passivity radius $\rho_{\M_T}(I)$
of the port-Hamiltonian system $\M_T:=\{J-R,G-K,G^{\mathsf{H}}+K^{\mathsf{H}},S+N\}$, and this for both the spectral and Frobenius norm.
\end{proof}

We point out that Theorem~\ref{thm:maxoverX} applies to all matrices in $\XWpd$, and therefore also to all matrices in $\XWpdpd$, which can be distinguished as follows.
\begin{corollary} \label{distinction}
The maximal value $\xi^*(X)$ of a matrix $X\in \XWpd$ for a given model $\M$ equals zero if $X$ is a boundary point of $\XWpd$ and is strictly positive if and only if $X$ is in $\XWpdpd$.
\end{corollary}
\begin{proof}
If $X$ is a boundary point of $\XWpd$ then $\det W(X,\M)=0$ and for those $X$, we have $\xi^*(X)=0$.
If $X$ belongs to $\XWpdpd$, then $W(X,\M)>0$ and $\diag(X,I_m)>0$.
Therefore there exists a $\xi>0$ such that $W(X,\M)>\xi\diag(X,I_m)$, and hence $\xi^*(X)>0$. Conversely, if
$\xi^*(X)>0$ then $W(X,\M)>0$ and $X\in \XWpdpd$.
\end{proof}
In order to maximize the passivity radius, it is clear that we need to look at $\XWpdpd$. For a given $X$ in $\XWpdpd$, we therefore consider the passivity LMI $W(X,\M_\xi)$ for the modified model $\M_\xi :=\{A+\frac{\xi}{2} I_n,B,C,D-\frac{\xi}{2} I_m\}$ with a $\xi$ chosen such that
\begin{equation} \label{shifted}
W(X,\M_\xi) := \left[\begin{array}{cc} -(A+\frac{\xi}{2} I_n)^{\mathsf{H}}X-X(A+\frac{\xi}{2} I_n) &   C^{\mathsf{H}}-XB \\    C-B^{\mathsf{H}}X & (D-\frac{\xi}{2} I_m)^{\mathsf{H}}+ (D-\frac{\xi}{2} I_m) \end{array}\right] \ge 0.
\end{equation}
We have the following Lemma.
\begin{lemma} \label{inclusion}
For every $X> 0$ in $\XWpdpd$ and any $0\le \xi_- < \xi_+  \le \xi^*(X)$, the  systems $\M_{\xi_-}$ and $\M_{\xi_+}$ are passive. Moreover, the whole solution set of
$W(X,\M_{\xi_+})\ge 0$ is included in the solution set of  $W(X,\M_{\xi_-})> 0$.
\end{lemma}
\begin{proof}
The LMIs for two different values $\xi_-<\xi_+$ are related as
\[
 W(X,\M_{\xi_+}) = W(X,\M_{\xi_-}) - (\xi_+-\xi_-) \diag(X,I_m).
\]
Since $X\in\XWpdpd$, $\xi^*(X)>0$ and $\diag(X,I_m)>0$, it follows that
\begin{equation} \label{ineqs}
W(X,\M) \ge W(X,\M_{\xi_-}) > W(X,\M_{\xi_+})\ge W(X,\M_{\xi^*(X)}) \ge 0.
\end{equation}
Hence, the systems  $\M_{\xi_-}$ and $\M_{\xi_+}$ are passive, since the associated LMIs have a nonempty solution set. Now consider {\em any} $X$ for which $W(X,\M_{\xi_+})\ge 0$. Since $\xi_+$ is strictly positive, so is $\xi^*(X)$ and hence $X\in \XWpdpd$. It then follows from \eqref{ineqs} that $W(X,\xi_-)>0$. Hence,
the solution set  of $W(X,\M_{\xi_+}) \ge 0$ is included in the solution set of $W(X,\M_{\xi_-}) > 0$.
\end{proof}
Lemma~\ref{inclusion} implies that for a given $X\in \XWpdpd$, the solution sets of  $W(X,\M_\xi) \ge 0$ are shrinking with increasing $\xi$.
It remains to find the $\xi^*(X)$ that corresponds to the largest possible passivity radius. We can obtain this value by relating it to the passivity of the transfer function
\[
  \T_\xi(s):=C((s-\xi/2)I_n -A)^{-1}B+(D-\xi I_m/2),
\]
of the modified system $\M_\xi$. Note that we have  assumed that the associated system is minimal, a property which is not changed by the shift.
It follows from the discussion of Section \ref{sec:PH} that this transfer function corresponds to a {\em strictly} passive system if and only if (i) the transfer function  $\T_\xi(s)$ is asymptotically stable and (ii) the matrix function $\Phi_\xi(s):=\T^\mathsf{H}_\xi(-s)+\T_\xi(s)$ is strictly positive on the $\imath \omega$ axis, with $\omega=\infty$ included.  It has been presented in Section \ref{sec:PH} that the zeros of $\Phi_\xi(s)$ are also the eigenvalues of the Hamiltonian matrix
\begin{equation*}
H_\xi :=\left[ \begin{array}{cc} A+\xi I_n/2 & 0 \\ 0 & -(A^{\mathsf{H}}+\xi I_n/2) \end{array}\right] + \left[ \begin{array}{cc}  -B \\ C^{\mathsf{H}} \end{array}
\right] (D^{\mathsf{H}}+D-\xi I_m)^{-1} \left[ \begin{array}{cc} C & B^{\mathsf{H}} \end{array} \right],
\end{equation*}
provided that $D^{\mathsf{H}}+D-\xi I_m > 0$ and the realization of $\M_\xi$ is minimal.

The three algebraic conditions corresponding to strict passivity of $\T_\xi(s)$ are, therefore, given by
\begin{enumerate}
	\item [A1.] $A+\xi I_n/2$ has all its eigenvalues in the open left half plane (stability).
	\item [A2.] $D^{\mathsf{H}}+D-\xi I_m$ has strictly positive eigenvalues (positive-realness at $\omega=\infty$).
	\item [A3.] $H_\xi$ has no eigenvalues on the $\imath\omega$ axis  (positive-realness at finite $\omega$).
\end{enumerate}
All of these conditions are phrased in terms of eigenvalues of certain matrices that depend on the parameter $\xi$. Since eigenvalues are continuous functions of the matrix elements, one can consider limiting cases for the above conditions. As explained in Section \ref{sec:PH}, the passive transfer functions are limiting cases of the strictly passive ones. These limiting cases correspond to the
first value of $\xi$ where one of the three algebraic conditions fails. Note that condition A3. is more robustly expressed in terms of the eigenvalues of the matrix pencil
\begin{equation} \label{statespaceXi}
S_\xi(s) :=
\left[ \begin{array}{cc|c} 0 & A+\xi I_n/2-sI_n & B \\
	A^{\mathsf{H}}+\xi I_n/2+sI_n & 0 & C^{\mathsf{H}} \\ \hline B^{\mathsf{H}} & C & D^{\mathsf{H}}+D-\xi I_m  \end{array} \right].
\end{equation}

It is obvious  that the conditions A1.-A3. are not satisfied anymore for large enough $\xi$. For instance, for $\xi > \lambda_{\min}(D^{\mathsf{H}}+D)$ the second condition fails and $\lambda_{\min}(D^{\mathsf{H}}+D)$ is thus a simple upper bound for $\xi^*(X)$ for any $X$.
\begin{theorem} \label{thm:Xi} Let $\M$ be a strictly passive system. Then there is a bounded supremum $\Xi:=\sup_\xi \{\xi \; | \; \T_\xi(s) \mathrm{\; is \; strictly \; passive}\}$, for which the following properties hold
\begin{enumerate}
\item  $\T_\Xi(s)$ is passive,
\item the solution set of $W(X,\M_\Xi)\ge 0$ is not empty,
\item the solution of $W(X,\M_\Xi)> 0$ is empty,
\item for any $\xi < \Xi$ the solution set of $W(X,\M_\xi)> 0$ is non-empty,
\item  $\Xi:=\sup_X \xi^*(X)$ for all $X\in \XWpd$.
\end{enumerate}
\end{theorem}
\begin{proof}
	The existence of a bounded supremum follows from the fact that $\T_\xi(s)$ is strictly passive
	only if $\xi$ is smaller than $\lambda_{\min}(D^{\mathsf{H}}+D)$. Property 1. holds because
	$\T_\Xi(s)$ is the limit of $\T_\xi(s)$ for $\xi \rightarrow \Xi$. Property 2. is a direct consequence of the previous property. Property 3. follows by contradiction, since if $W(X,\M_\Xi)> 0$ would not be empty, then $\xi^*(X)$ for $X$ in the solution set of $W(X,\M_\Xi)> 0$, would be larger than $\Xi$. Property 4. follows from Lemma~\ref{inclusion}, where we use any $X$ in the solution set of $W(X,\M_\Xi)\ge  0$ and choose $\xi_+=(\Xi+\xi)/2$ and $\xi_-=\xi$ to show that $X$ also lies in the solution set of $W(X,\M_\xi)> 0$. Property 5. follows from
	$\xi^*(X)=\max \{ \xi \; | \; W(X,\M_\xi)\ge 0 \}$, which expresses that $\T_\xi(s)$ is passive.
\end{proof}
We now link the value of $\Xi$ in Theorem~\ref{thm:Xi} to the passivity radius of an {\em optimally robust} pH realization.
\begin{theorem} \label{thm:optimal}
	Let $\M:=\{A,B,C,D\}$ be a minimal realization of a strictly passive transfer function $\T(s):= C(sI-A)^{-1}B+D$.
	Then
\[
\Xi:=\sup_\xi \{\xi \; | \; \T_\xi(s) \mathrm{\; is \; strictly \; passive}\}
\]
 is the largest possible passivity radius out of all realizations of this transfer function. The models with such an optimal passivity radius correspond to a solution $X$ of $W(X,\M_\Xi)\ge 0$ and a  pH realization is given by $\M_T:=\{T^{-1}AT,BT,T^{-1}C,D\}$ where $X:=T^\mathsf{H}T$.
\end{theorem}
\begin{proof}
Consider realizations  $\M_T:=\{T^{-1}AT,BT,T^{-1}C,D\}$ with $X:=T^\mathsf{H}T$ and $X \in W(X,\M)\ge 0$. It was shown in Theorem~\ref{thm:maxoverX} that such realizations have passivity radius equal to $\xi^*(X)$ and Theorem~\ref{thm:Xi} shows that the supremum of all $\xi^*(X)$ is precisely $\Xi$.

Let us now consider an arbitrary model $\M$. Then its passivity radius is $\rho_\M=\rho_\M(X)$ for some $X \in \XWpdpd$. 	It follows that the passivity radius of the pH realization $\M_T$ derived from $X=T^\mathsf{H}T$ is larger than or equal to that of $\M$. Moreover, the corresponding passivity radius is $\xi^*(X)$. To complete the proof we point out that the matrices $X$ that maximize $\xi^*(X)$ are in $W(X,\M_\Xi)\ge 0$.
\end{proof}
In this section we have derived a characterization of the passivity radius of a strictly passive system. In the next section we show how this can be computed numerically.
\section{Computing the optimal passivity radius} \label{sec:computing}

In this section we describe algorithms that compute, within a given tolerance
$\tau$, an approximation of the optimal $\Xi$ as in Theorem~\ref{thm:optimal} for a given minimal realization $\M:=\{A,B,C,D\}$ of a passive system.

First of all, if  $\M$ is passive but not strictly passive then $\Xi=0$. If $\M$ is strictly passive, then
simple upper bounds for $\Xi$ are given by the conditions A1. and A2. in Section \ref{sec:maxpass}, i.e.,
\[
 \Xi_{up} = \min \left[  \min_j (-\Re\lambda_j(A)), \min_i\lambda_i(D^\mathsf{H}+D) \right].
\]
The procedure to check passivity is then to verify condition A3. for $0 \le \xi \le \Xi_{up}$ , namely that $S_\xi$, (or $H_\xi$) has no generalized eigenvalues on the imaginary axis.
We therefore first recall some basic properties of the scalar function
\begin{equation}\label{defgamma}
\gamma(\xi,\omega):=\lambda_{\min}\Phi_\xi(\imath\omega), \quad \mathrm{where} \quad
\Phi_\xi(s):= \T_\xi^\mathsf{H}(-s)+ \T_\xi(s).
\end{equation}
\begin{theorem}\label{conthm}
The real function $\gamma(\xi,\omega):=\lambda_{\min}\Phi_\xi(\imath\omega)$ in (\ref{defgamma}) is continuous in the real variables $\xi$ and $\omega$ and it has the following properties. It is positive for all $\omega$ in the interval $\xi\in[0,\Xi)$, it is zero for at least one value of $\omega$ at $\xi=\Xi$, and it is negative in some open interval(s) of $\omega$ whenever $\xi \in (\Xi,\Xi_{up}]$, provided that $\Xi<\Xi_{up}$.
\end{theorem}
\begin{figure}[h] \label{Fig:Frplot}
	\centering
	\caption{Three frequency plots for the cases $\xi$ smaller, equal and larger that $\Xi$}
	\includegraphics[height=8cm]{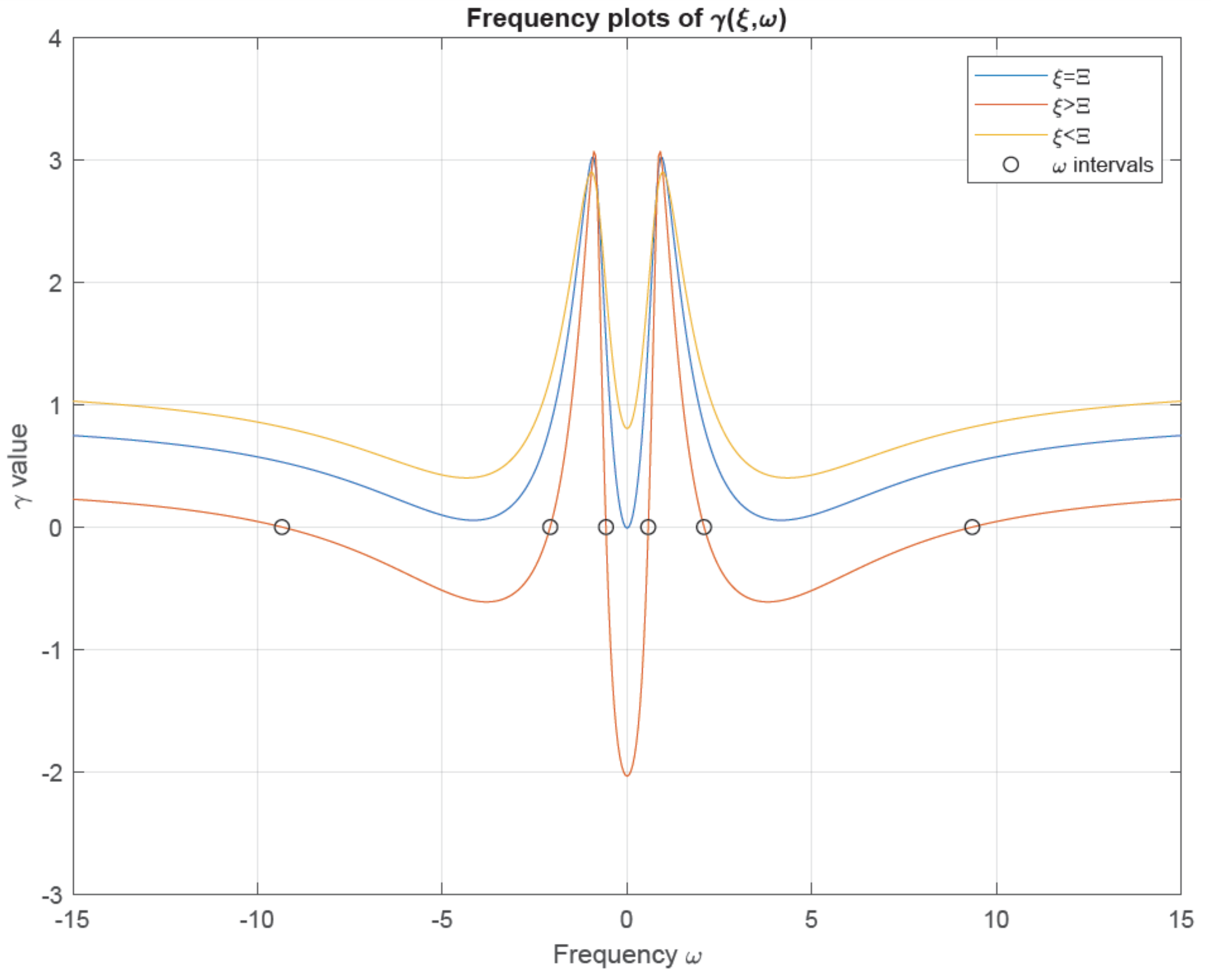}
\end{figure}%
\begin{proof}
	The continuity follows trivially from the fact that eigenvalues of a (Hermitian) matrix are continuous functions of the parameters of the matrix. It is clear that if $\xi<\Xi$ then the transfer function $\Phi(\imath\omega)$ is positive definite for all $\omega$ and so is its smallest eigenvalue $\gamma(\xi,\omega)$. When we increase $\xi$ and reach the limiting value $\Xi$, then the transfer function is passive but not strictly passive anymore, and hence $\gamma(\xi,\omega)$ must loose strict positivity in at least one point $\imath\omega$, \ie $S_\xi$ must must have at least one eigenvalue on the imaginary axis. When we further increase $\xi$,  $S_\xi$ will have more purely imaginary eigenvalues, and there will be purely imaginary eigenvalues all the way up to $\Xi_{up}$ because at $\omega=\pm\infty$ the function eigenvalues of $\Phi_\xi(\pm\infty)$ are those of  $D^\mathsf{H}+D-\xi I_m$. By continuity, $\gamma(\xi,\omega)=0$ must therefore intersect the zero-level for all $\xi\in[\Xi,\Xi_{up}]$.
	Notice also that when $D^\mathsf{H}+D-\xi I_m>0$, then the number of generalized eigenvalues on the imaginary axis is bounded by $2n$, since the pencil $S_\xi$ is then regular and has $m$ infinite generalized eigenvalues.
 	These three different cases are also depicted in Figure \ref{Fig:Frplot}.
\end{proof}	
As a consequence of Theorem~\ref{conthm}, the smallest value of $\xi$ in the interval $[0,\Xi_{up}]$, where condition A3. fails is equal to $\Xi$. (Note that this could be equal to $\Xi_{up}$.) One can then apply bisection to this interval and check the presence of purely imaginary eigenvalues in the above interval. Putting $\Xi_{lo}=0$, we have the following procedure.\\
\\
\emph{\bf Bisection procedure for computing $\Xi$:}
\[
 \xi = (\Xi_{lo} + \Xi_{up})/2, \; \mathrm{if\;} H_\xi \mathrm{\; has \; purely \; imaginary \; eigenvalues\;} \Xi_{up}:=\xi, \mathrm{\; else\; }  \Xi_{lo}:=\xi.
\]
Since the interval for $\Xi$ shrinks by a factor $2$ at each step of the iteration, then in $k=\lceil \log_2(\Xi_{up}/\tau) \rceil$ steps we will have $\Xi_{up}-\Xi_{lo}\leq\tau$.	

One can also make use of the computed eigenvalue decompositions of
\[
S_\xi(\imath \omega)=\left[ \begin{array}{ccc} 0 & A+\xi I_n/2-\imath \omega I_n & B \\
A^{\mathsf{H}}+ \xi I_n/2+\imath \omega I_n & 0 & C^{\mathsf{H}} \\
B^{\mathsf{H}} & C & D^{\mathsf{H}}+D-\xi I_m \end{array} \right],
\]
to construct a method with faster convergence properties.
This approach relies on the fact that $\gamma(\xi,\omega)$ is a continuous
function with smoothness properties at the local minima. It is inspired from a method in \cite{BoyB90} for the computation of the $L_\infty$-norm of a transfer function.

\begin{figure}[h] \label{Fig:Smooth}
	\centering
	\caption{3D illustration of $\gamma(\xi,\omega)$.}
	\includegraphics[height=9cm]{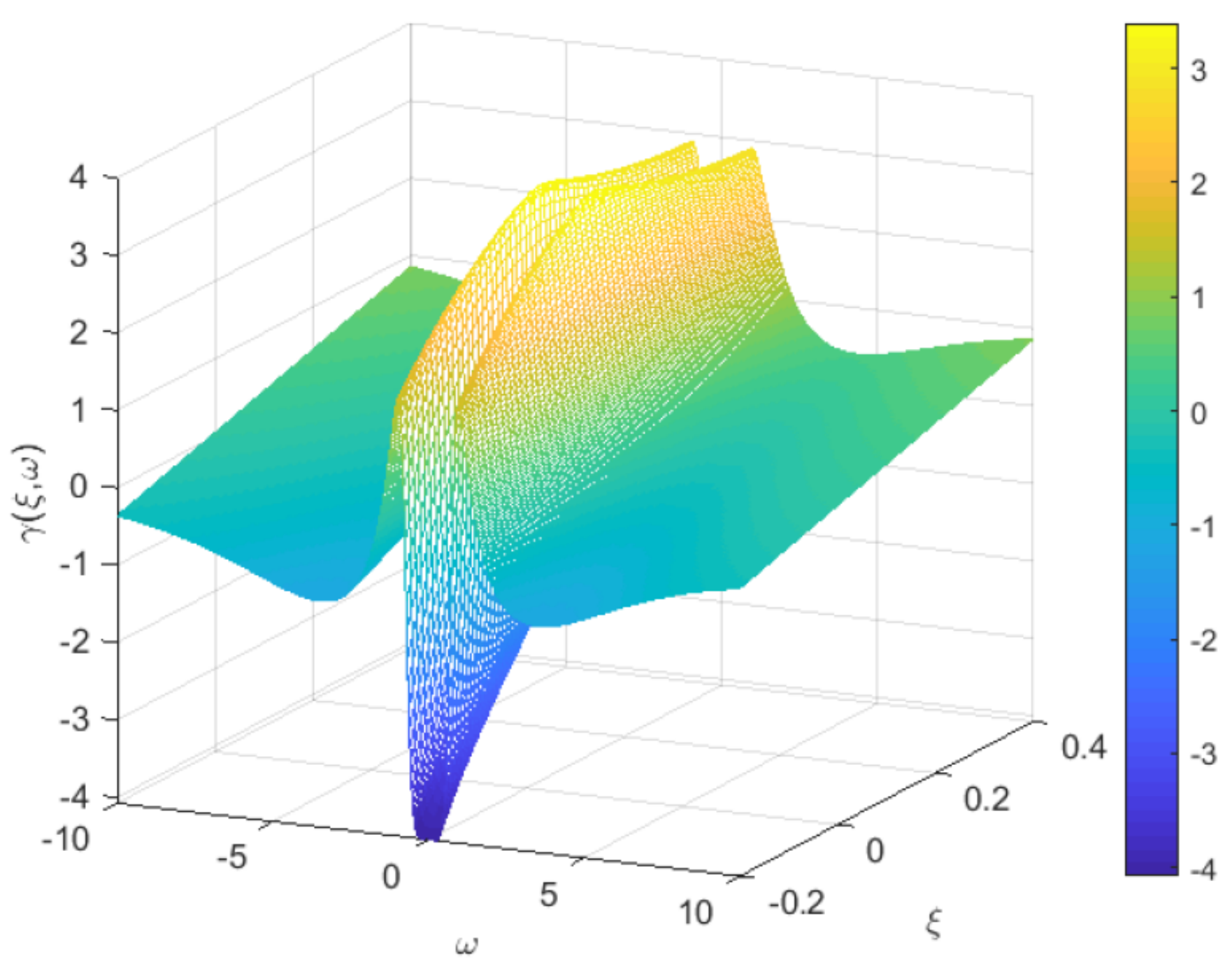}
\end{figure}%
%

Let $\hat \xi \in[\Xi,\Xi_{up}]$, let $\Omega$ be the set of purely imaginary eigenvalues of $S_{\hat\xi}$, let $\partial\Omega$ be the corresponding set of derivatives of these eigenvalues with respect to $\omega$ , and let $\gamma(\hat \xi,\omega)$ be the smallest eigenvalue of $\Phi_{\hat \xi}(\imath\omega)$. Then we exploit the following ideas.
\begin{enumerate}
	\item The real roots $\omega_i$ of $\gamma(\hat \xi,\omega)$ are a subset of $\Omega$. They can be identified by looking at the derivatives in $\partial\Omega$. A detailed description of this selection procedure is described in \cite{SreVT96}.
	\item For a given value of $\hat \xi$, $\gamma(\hat \xi,\omega)$ we can locally approximate $\gamma$ in the neighborhood of a local minimum $\omega^*$ by a quadratic approximation $\gamma^*+(\omega-\omega^*)^2$.
	A detailed description of this result is described in \cite{BoyB90}.
	\item If $\omega_1 < \omega_2$ are two ``consecutive" zeros of $\gamma(\hat\xi,\omega)$ then the smallest real root $\tilde \xi$ of $S_\xi(\hat\omega)$ lies between $0$ and $\hat\xi$ in that interval. This follows from the fact that the smallest eigenvalue $\gamma(\hat \xi,\omega)$ for the current value $\hat \xi$ is negative in that interval of $\omega$. Moreover, at the midpoint $\hat\omega:=(\omega_1+\omega_2)/2$
	the smallest real root $\tilde \xi$ is an improved upper bound for $\Xi$, in the sense that close to $\Xi$, the corresponding error improves quadratically, i.e., $|\Xi-\tilde \xi| \approx |\Xi-\hat \xi|^2$. A detailed description of this result follows from \cite{BoyB90} and the continuity of the two-variable function $\gamma(\xi,\omega)$.
\end{enumerate}

Since we can compute the zeros $\omega_i$ of $\gamma(\hat \xi,\omega)$) for a given value of $\hat\xi$, and we can find the smallest real root $\tilde \xi$ of $S_\xi(\imath \hat \omega)$ for a given value of $\hat \omega$, we can then obtain an algorithm with improved local convergence for computing $\Xi$ based on the above ideas. A detailed analysis of its convergence requires more investigation and will be considered in forthcoming work.

\newpage

\emph{\bf An improved algorithm:}
\begin{enumerate}
\item $\hat \xi := \Xi_{up}-\tau$.
\item Compute the eigenvalues of $S_{\hat\xi}(\imath\omega)$ and select those corresponding to real zeros of $\gamma(\hat \xi,\omega)$.
\item {\bf if} $\gamma(\hat \xi,\omega)$ has no real zeros, {\bf then} $\Xi_{lo}=\hat \xi$, {\bf stop},
\item {\bf else} let $\hat\omega:=(\omega_1+\omega_2)/2$ for the largest interval $[\omega_1,\omega_2]$ of these roots. \\ \hspace*{.7cm} Compute the real eigenvalues $\xi_i$ of
$S_\xi(\imath \hat \omega)$ and update $\Xi_{up} :=\min_i \xi_i$. \\ \hspace*{.7cm} Define the next guess $\hat\xi := \Xi_{up}-\tau$ and go to step 2.\\
\end{enumerate}
This algorithm will typically require only a few iterations and stops with an interval $[\Xi_{lo},\Xi_{up}]$ of size $\tau$. Each step of both, the bisection and the faster converging algorithms has a complexity that is cubic in the matrix dimensions, but they are guaranteed to provide a required accuracy. For large scale problems, this complexity may become a problem, but one can combine it with iterative structured Krylov space techniques like those discussed in \cite{MehSS12} or model reduction based techniques as in \cite{AliBMSV17,AliMM18_ppt}.

\section{The distance to passivity} \label{sec:distance}
In real world applications the system data are usually inaccurate and
therefore only an approximate model is known. Often the real physical problem is nonlinear or a partial differential equation and it is approximated by a finite element or finite difference model \cite{IdaB97}, it may be obtained by a realization or system identification procedure \cite{CoePS99,GusS01,SarAN03},  or  the result of a model reduction procedure \cite{Ant05}. Then it is in general not clear that in the approximation process passivity is preserved. In this situation one approximates the non-passive system by a (hopefully) nearby passive system, by introducing small or minimal perturbations to the model $\M:=\{A,B,C,D\}$, see
\cite{AlaBKMM11,BruS13,CoePS99,FreJ04,Gri04,SarAN03,SchS07}. Therefore, in this section we start with  a system $\M$ that is not passive and ask the question about the smallest perturbation $\Delta_\M$ of the model that makes the system $\M+\Delta_\M$ passive.

It is clear from our previous analysis that to find the smallest perturbation that makes the system passive is equivalent to finding the smallest perturbation $\Delta_\M$ such that the LMI $W(X,\M+\Delta M) \ge 0$ has a Hermitian and positive semidefinite solution $X$. Moreover, if the perturbed system remains minimal, then we expect $X>0$. We recall the notations
\[
\S := \left[ \begin{array}{rr}-A & -B \\ C & D
\end{array}
\right], \quad \mathrm{and} \quad \Delta_\S := \left[ \begin{array}{rr}-\Delta_A & -\Delta_B \\ \Delta_C & \Delta_D
\end{array}
\right].
\]
\begin{definition} The \emph{distance to passivity} of a minimal system $\M\!=\!\{A,B,C,D\}$ is the minimum
norm $\|\Delta_\S\|_{2} $ ($\|\Delta_\S\|_F$) such that there exists a matrix $X>0$ satisfying
\begin{equation} \label{makepassive}
(\S+\Delta_\S)^\mathsf{H}\hat X + \hat X (\S+\Delta_\S) \ge 0, \quad \mathrm{where} \quad \hat X:= \left[ \begin{array}{cc} X & 0 \\ 0 & I_m \end{array} \right].
\end{equation}
\end{definition}
In the following we need an extension of Lemma \ref{inclusion}, for which we consider the LMI for the modified model $\M_{-\xi} :=\{A-\frac{\xi}{2} I_n,B,C,D+\frac{\xi}{2} I_m\}$ with the corresponding transfer function
\begin{equation} \label{Tminusxi}
T_{-\xi}(s) := C((s+\xi/2)I_n -A)^{-1}B+(D+\xi I_m/2)
\end{equation}
and LMI

\begin{eqnarray} \nonumber
W(X,\M_{-\xi}) &:=& \left[\begin{array}{cc} -(A-\frac{\xi}{2} I_n)^{\mathsf{H}}X-X(A-\frac{\xi}{2} I_n) &   C^{\mathsf{H}}-XB \\    C-B^{\mathsf{H}}X & (D+\frac{\xi}{2} I_m)^{\mathsf{H}}+ (D+\frac{\xi}{2} I_m) \end{array}\right] \\
&\ge& 0.\label{shifted2}
\end{eqnarray}
\begin{lemma} \label{inclusion2}
Let $\M:=\{A,B,C,D\}$ be a non-passive system. Then for every $X> 0$ in $\mathbb{H}_n$ there exists a $\xi^*(X)> 0$ such that the LMI \eqref{shifted2} for the system $\M_{-\xi^*(X)}$ holds. Moreover, for every value
$\xi > \xi^*(X)$, the system $\M_{-\xi}$ is passive.
\end{lemma}
\begin{proof}
Clearly we have that $W(X,\M_{-\xi})=W(X,\M) +\xi \hat X$. Since $W(X,\M)$ is bounded from below and $\hat X>0$, the inequality $W(X,\M_{-\xi})\ge 0$ holds for a sufficiently large value of $\xi$. Let $\xi^*(X)$ be the smallest $\xi$ value for which $W(X,\M_{-\xi})\ge 0$ holds, then $W(X,\M_{-\xi})=W(X,\M_{-\xi^*(X)})+(\xi-\xi^*(X))\hat X$, which implies that the passivity condition holds for all $\xi > \xi^*(X)$.
\end{proof}

To find the optimal $\xi$, let us consider first perturbations $\Delta_\S$ that are a multiple of the identity.
\begin{theorem}
The minimum norm perturbation of the type
\begin{equation} \label{restrict}
\Delta_\S= \frac12 \xi I_{n+m}
\end{equation}
that makes system $\M$ passive, has spectral norm $\Xi/2$ and Frobenius norm $\Xi\sqrt{n+m}/2$, where $\Xi$ is the minimum value of $\xi$ such that the model $\M_{-\xi}:=\{A-\xi I_n/2,B,C,D+\xi I_m/2\}$
with transfer function $T_{-\xi}(s)$ is passive.
\end{theorem}
\begin{proof}
It follows from \eqref{makepassive} that $\xi$ must satisfy
\begin{equation} \label{minusxi}
(\S+\frac12 \xi I_{n+m})^\mathsf{H}\hat X + \hat X (\S+\frac12 \xi I_{n+m}) \ge 0, \quad \mathrm{where} \quad \hat X:= \left[ \begin{array}{cc} X & 0 \\ 0 & I_m \end{array} \right],
\end{equation}
for some $X>0$. By Lemma~\ref{inclusion2} there exists a bounded minimal solution, which we call $\Xi$. The state-space model corresponding to $\S+\frac12 \xi I_{n+m}$ is $\M_{-\xi}$
with transfer function \eqref{Tminusxi}. Therefore, $\Xi$ is the
smallest value of $\xi$ that makes the model $\M_{-\xi}$ with transfer function $T_{-\xi}(s)$ become passive. We can then choose $X>0$ from the solution set of $W(X,\M_{-\Xi})\ge 0$ to satisfy \eqref{makepassive}.
\end{proof}
The minimal value $\Xi$ of $\xi$ for the restricted class of perturbations being a multiple of the identity can then be computed with the algorithms described in the previous section.

Since, most likely,  the perturbation will not lead to  \eqref{makepassive} being an equality, we can  reduce the Frobenius norm of the perturbation $\Delta_\S$ by considering more general perturbations than
\eqref{restrict}.
In order to do that, we use a matrix $X$ from the set $W(X,\M_{-\Xi})\ge 0$, where $\Xi$ was obtained from perturbations as in \eqref{restrict}. We use the factorization $X=T^\mathsf{H}T$, to transform the system $\M$ to the pH form $\M_T=\{TAT^{-1},TB,CT^{-1},D\}:=\{J-R,G-K,G^\mathsf{H}+K^\mathsf{H},S+N\}$. If we denote the transformed matrix $\S$ as $\hat \S:=\hat T\S\hat T^{-1}$, where $\hat T=\diag(T,I_m)$, then it follows that
\[
  \frac12(\hat \S^\mathsf{H}+\hat\S)=:\hat {\mathcal R}=\left[ \begin{array}{cc} R & K \\ K^\mathsf{H} & S \end{array} \right] \ge -\Xi I_{n+m},
\]
since $\M_{-\Xi}$ is passive. The smallest eigenvalue of $\frac12(\hat \S^\mathsf{H}+\hat\S)$ is greater or equal than $-\Xi $, but the other eigenvalues may be larger or even positive.

To obtain a solution for the minimum Frobenius norm problem
\[
 \min_{\Delta_\S}\{\|\Delta_\S\|_F \; | \; \frac12[(\hat T\Delta_\S\hat T^{-1})^\mathsf{H}+\hat T\Delta_\S\hat T^{-1}] + \left[ \begin{array}{cc} R & K \\ K^\mathsf{H} & S \end{array} \right]\ge 0\},
 \]
let $\hat T=\hat U^T \hat \Sigma \hat V $ be a singular value decomposition with $\hat \Sigma$ diagonal and $\hat U,\hat V$ unitary. Then we can alternatively study the problem
\begin{equation} \label{alternative}
\min_{\Delta_{\tilde \S}}\{\|\Delta_{\tilde \S}\|_F \; | \; \frac12[(\hat \Sigma \Delta_{\tilde \S}\hat \Sigma^{-1})^\mathsf{H}+\hat \Sigma \Delta_{\tilde \S}\hat \Sigma^{-1}] + \tilde {\mathcal R}\ge 0\},
\end{equation}
where $\Delta_{\tilde \S}:=\hat V \Delta_\S \hat V^H$ and $\tilde {\mathcal R}:= \hat U \hat {\mathcal R} \hat U^\mathsf{H}$.

For the given $\tilde {\mathcal R}$ we have a unitary spectral decomposition
\[\tilde U^\mathsf{H} \tilde{\mathcal R} \tilde U=\diag (D_1, 0)-\diag (0,D_2),
\]
where $D_1$ is diagonal with positive diagonal elements and $D_2$ is diagonal with nonnegative diagonal elements. Then $\Delta \tilde {\mathcal R} =\tilde U^\mathsf{H} \diag(0, D_2) \tilde U$ is the minimum Frobenius norm perturbation so that $\tilde {\mathcal R}+\Delta \tilde {\mathcal R}$ is positive semidefinite. We can then replace the optimization problem \eqref{alternative} by the
simpler problem
\begin{equation} \label{simpler}
\min_{\Delta_{\tilde \S}}\{\|\Delta_{\tilde \S}\|_F \; | \; \frac12[(\hat \Sigma \Delta_{\tilde \S}\hat \Sigma^{-1})^\mathsf{H}+\hat \Sigma \Delta_{\tilde \S}\hat \Sigma^{-1}] =\Delta \tilde {\mathcal R}\}.
\end{equation}
Indeed, let $\Delta F$ be the strictly lower triangular part plus $\frac 12$ the diagonal part of
$\Delta \tilde {\mathcal R}$. Then clearly $\hat \Sigma^{-1} \Delta F \hat \Sigma$ is the optimal Frobenius norm solution to \eqref{simpler}.

It should be pointed out that we could have used another matrix $X$ from the domain of $W(X,\M_{-\Xi})\ge 0$.
The matrices $R$, $K$ and $S$ will then change and the minimum Frobenius norm solution will be affected as well. Note that if we choose a matrix $X$ that does {\em not} belong to the solution set of $W(X,\M_{-\Xi})\ge 0$, then because of Lemma \ref{inclusion2}, the value of the most negative eigenvalue in $\Lambda$ will be equal to $-\xi^*(X)$, but it follows that $\xi^*(X) > \Xi$. It is therefore unlikely that the Frobenius norm of the constructed minimum norm  perturbation will be smaller.

\section{The distance to stability} \label{sec:stability}

We can employ similar arguments as in the previous section for the computation of the smallest perturbation that makes a given system stable, see \cite{GilMS18,GilS17,OrbNV13} for other approaches. When searching for the smallest perturbation we only have to study the $(1,1)$ block of the LMI (\ref{KYP-LMI}), i.e. the case where the matrices $B$, $C$ and $D$ are void, or where $m=0$. In this case we write the LMI (\ref{KYP-LMI}) as
\begin{equation} \label{stabmodel}
W(X,A)= -A^\mathsf{H}X-XA = 2R \geq 0
\end{equation}
while the shifted LMI takes the form
\begin{equation} \label{stabLMI}    W_{\xi}(X,A):= W(X,A)-\xi X = -(A+ \xi I_n/2)^\mathsf{H}X-X(A+\xi I_n/2) = 2R-\xi X \ge 0.
\end{equation}

For any $X$ in the solution set of $W(X,A) \ge 0$, the shifted LMI \eqref{stabLMI} has a nonempty solution set as long as $\xi$ is smaller than $\xi^*(X)$.
Instead of the three conditions A1.--A3. for passivity, we only have one condition, namely that the Hamiltonian matrix
\begin{equation} \label{stabHamiltonian}   H_\xi :=\left[ \begin{array}{cc} A+\xi I_n/2 & 0 \\ 0 & -(A^{\mathsf{H}}+\xi I_n/2) \end{array}\right] =
H + \xi/2\left[ \begin{array}{cc} I_n & 0 \\ 0 & -I_n \end{array}\right]
\end{equation}
has no purely imaginary eigenvalue for all $0 \le \xi < \xi^*(X)$. So if $\Xi$ is the largest value of $\xi^*(X)$ over all
$X$ in the solution set of $W(X,A)\ge 0$, then the solution set of  $W_\Xi(X,A)\ge 0$ is not empty, but it has an empty interior.

A well-known formula for the \emph{stability radius} $\rho_A$ of a given matrix $A$ is given by
\[
    \rho_A = \min_{\omega\in \mathbb{R}} \sigma_{\min}(A-\imath \omega I_n).
\]
The perturbation that achieves the minimum is constructed from the smallest singular value $\sigma_{\min}$ and corresponding singular vectors $u$ and $v$ of $A-\imath\omega I_n$ at the minimizing value of $\omega$~:
\[
 (A-\imath \omega I_n)v=\sigma_{\min}u, \quad \Delta_A := \sigma_{\min}uv^\mathsf{H}.
\]
A similar result is obtained using the analysis of Section~\ref{sec:passrad}. This encourages us to also use the distance to passivity results to derive bounds for the \emph{distance to stability}.

Let us first look at the solution of the problem for a specific diagonal form of the perturbation. Given an
unstable matrix $A$ we study the question what is the smallest perturbation of the type $\Delta_A:=-\frac12\xi I_n$ that makes the matrix stable. Following a similar analysis as in Section \ref{sec:distance} we can prove the following theorem.
\begin{theorem}
The minimum norm solution of the form
\begin{equation} \label{restrictA}
\Delta_A= -\frac12 \xi I_{n}
\end{equation}
for the stabilization problem of a matrix $A$, has spectral norm $\Xi/2$ and Frobenius norm $\Xi\sqrt{n}/2$, where $\Xi$ is the minimum value of $\xi$ such that the matrix $A_{-\xi}:=A-\xi I_n/2$, is stable.
\end{theorem}
\begin{proof}
Clearly $\xi$ must satisfy
\begin{equation} \label{minusxiA}
-(A-\frac12 \xi I_{n})^\mathsf{H} X - X (A-\frac12 \xi I_{n}) \ge 0
\end{equation}
for some $X>0$. It follows from Lemma \ref{inclusion2} that there exists a bounded minimal solution, which we call $\Xi$.
The matrix corresponding to $A_{-\xi}:=A-\frac12 \xi I_{n}$ has the associated Hamiltonian matrix
\[
 H_{-\xi} := \left[ \begin{array}{cc} A-\xi I_n/2 & 0 \\ 0 & -(A^{\mathsf{H}}-\xi I_n/2) \end{array}\right].
\]
Let $\Xi$ be the smallest value such that $H_{-\xi}$ has no purely imaginary eigenvalues for $\xi <\Xi$.
Then the matrix $A_{-\Xi}$ is stable (but not asymptotically stable) and we can then choose $X>0$ from the solution set of $W(X,A_{-\Xi})\ge 0$ to satisfy \eqref{minusxiA}.
\end{proof}

Since we made some of the eigenvalues of the LMI in \eqref{minusxiA} strictly positive, rather than non-negative,
we can further reduce the Frobenius norm of the perturbation $\Delta_A$ when removing the imposed restriction of a diagonal perturbation.
In order to do that, we use a matrix $X$ from the set $W(X,A_{-\Xi})\ge 0$, where $\Xi$ was obtained from the diagonal perturbation.
Using its factorization $X=T^\mathsf{H}T$, we transform $A$ to $A_T=TAT^{-1}:=J-R$, which satisfies
\[  -\frac12(A_T^\mathsf{H}+A_T)= R \ge -\Xi I_{n},
\]
since $A_{-\Xi}$ is stable. We know that the smallest eigenvalue equals $-\Xi $, but the others may be larger or even positive.
In order to construct a nearly optimal solution to this, we can again use the two stage procedure and the perturbation result derived in the previous section.

\section{Conclusion}\label{sec:concl}

We have presented analytic formulas and numerical methods to construct optimally robust port-Hamiltonian realizations of a given transfer function of a linear time invariant passive system. We have shown how to use shifted linear matrix inequalities to achieve this goal. The techniques can also be applied to compute a nearby passive system to a given non-passive one or a nearby stable system to a given unstable one.

\bibliographystyle{siamplain}
%

\end{document}